\documentclass[12pt]{amsart}
\makeatletter
\@namedef{subjclassname@2020}{\textup{2020} Mathematics Subject Classification}
\makeatother

\usepackage{amssymb}
\usepackage{graphicx}
\usepackage{enumerate}

\newtheorem{theorem}{Theorem}[section]
\newtheorem{lemma}[theorem]{Lemma}
\newtheorem{proposition}[theorem]{Proposition}
\newtheorem{corollary}[theorem]{Corollary}

\theoremstyle{definition}
\newtheorem{definition}[theorem]{Definition}
\newtheorem{example}[theorem]{Example}

\theoremstyle{remark}
\newtheorem{remark}[theorem]{Remark}

\numberwithin{equation}{section}

\newcommand{\ph}{\varphi}
\newcommand{\eps}{\varepsilon}
\newcommand{\R}{\mathbb{R}}
\newcommand{\Z}{\mathbb{Z}}
\newcommand{\N}{\mathbb{N}}

\newcommand{\ls}{\underline{\sigma}}
\newcommand{\us}{\overline{\sigma}}
\newcommand{\lt}{\underline{\tau}}
\newcommand{\ut}{\overline{\tau}}
\newcommand{\oB}{\overline{B}}

\newcommand{\Rak}{\mathbb{R}_{(\alpha, K)}}
\newcommand{\Rakp}{\mathbb{R}_{(\alpha, K')}}

\newcommand{\ee}{\hfill$\Diamond$\end{example}}

\newcommand{\ak}{$(\alpha,K)$}
\newcommand{\akp}{$(\alpha,K')$}
\newcommand{\akch}{$(\alpha,K)$-chained}
\newcommand{\akpch}{$(\alpha,K')$-chained}
\newcommand{\akcomp}{$(\alpha,K)$-component}
\newcommand{\akconn}{$(\alpha,K)$-connected}
\newcommand{\akpconn}{$(\alpha,K')$-connected}
\newcommand{\aconn}{$\alpha$-connected}
\newcommand{\akray}{$(\alpha,K)$-ray}
\newcommand{\akpray}{$(\alpha,K')$-ray}
\newcommand{\akqray}{$(\alpha,K)$-quasi-ray}

\newcommand{\akend}{$(\alpha,K)$-end}
\newcommand{\akpend}{$(\alpha,K')$-end}
\newcommand{\aend}{$\alpha$-end}

\newcommand{\calek}{{\mathcal E}^\alpha_K}
\newcommand{\calekp}{{\mathcal E^\alpha}_{K'}}
\newcommand{\caleinf}{{\mathcal E}^\alpha_\infty}

\begin{document}

\title{Sparse metric spaces and sparse ends}

\author{William Geller}
\address{Department of Mathematical Sciences, IU Indianapolis, 402 N. Blackford
  Street, Indianapolis, IN 46202}
\email{wgeller@iu.edu}

\author{Micha{\l} Misiurewicz}
\address{Department of Mathematical Sciences, IU Indianapolis, 402 N. Blackford
  Street, Indianapolis, IN 46202}
\email{mmisiure@iu.edu}


\date{June 9, 2026}

\begin{abstract}
We study metric spaces that in some sense thin out at infinity. We
define and investigate a measure of sparsity that is a quasi-isometry
invariant, and introduce an analogue of topological ends for sparse
spaces that is also invariant under quasi-isometries. We study some
examples arising in various contexts.
\end{abstract}

\maketitle

\section{Introduction}\label{intro}

We study metric spaces that in some sense thin out at infinity. While
our motivation comes from the search for invariants in the theory of
coarse dynamical systems (\cite{GM1}), such spaces arise in diverse areas,
including the study of Roe algebras (\cite{BCL}, \cite{BFV}, \cite{CMS}, \cite{WW}),
substitutions and other symbolic dynamical systems (\cite{LM}), and number
theory (\cite{M}, \cite{Z}).

Our primary goals here are to define and analyze some closely related
notions of thinning out at infinity, give a measure of sparsity that
is a quasi-isometry invariant, and define an analogue of topological
ends for sparse spaces that is also invariant under quasi-isometries.

Protasov \cite{P} has previously studied sparse metric spaces from a
different point of view. Although his definition of sparse spaces
differs from ours, his unbounded thin spaces are easily seen to be
equivalent to our sparse spaces, and his coarsely thin spaces coincide
with our quasi-sparse spaces. In particular, our Proposition~\ref{qs}
follows from his Theorem~5, and our Proposition~\ref{qs-ad0} from his
Theorem~1 and the remarks before it.

In Section 2, we define sparse metric spaces and the related notions
of quasi-sparse and finitely sparse spaces, introduce the order of
sparsity and observe that it is invariant under quasi-isometries,
study the relation of quasi-sparsity to asymptotic dimension, and
present a number of examples and connections to recent work in other
fields.

Section 3 draws a connection between sparsity of a metric space and
the notion of tryposity at infinity, which involves the degree to
which pairs of points in large balls fail to have midpoints. The subsequent
sections do not depend on it.

In Section 4 we define for each $\alpha \ge 0$ the space of \aend s of
a metric space, and show that a quasi-isometry of metric spaces
induces a homeomorphism on their spaces of \aend s. When $\alpha = 0$
this recovers the coarse ends studied in \cite{AC}, while for $\alpha > 0$
this gives a useful notion of ends for sparse spaces.

In Section 5 we give a number of examples of metric spaces and their
spaces of \aend s.

We thank Slawek Klimek and Jon Keating for useful suggestions and
references.

\section{Sparse metric spaces}\label{spaces}

Let $(X,d)$ be an unbounded metric space. Fix a base point $x_0\in X$.
Define a function $s:(0,\infty)\to\R$ as follows:
\begin{equation*}
s(R)=\inf\{d(x,y):x,y\in X\setminus B(x_0,R), x\ne y\},
\end{equation*}
where $B(x_0,R)$ is the open ball of radius $R$, centered at $x_0$.
Clearly, the function $s$ is nondecreasing. We will assume that
$\lim_{r\to\infty}s(R)=\infty$ and will call such spaces \emph{sparse
  at infinity}, or just {\em sparse}. This property does not depend on
the choice of $x_0$, since for $x_1,x_2\in X$ we have $B(x_2,R)\subset
B(x_1,R+d(x_1,x_2))$.

Now we define the \emph{lower and upper sparsity order at infinity}
($\ls$ and $\us$) as
\begin{equation*}
\ls(X)=\varliminf_{R\to\infty}\frac{\log s(R)}{\log R},\ \ \ \
\us(X)=\varlimsup_{R\to\infty}\frac{\log s(R)}{\log R}.
\end{equation*}
If $\ls(X)=\us(X)$, we will speak simply of the \emph{sparsity order
  at infinity} or just the \emph{sparsity order} and denote it
$\sigma(X)$.

These similarly do not depend on the choice of $x_0$.

Let us give some examples. In these examples
$X=\{0\}\cup\{a_n:n=1,2,3,\dots\}\subset\R$.

\begin{example}\label{n^q}
For each $q>1$, define $X_q$ by taking $a_n=n^q.$ In order to compute
$\ls(X_q)$ we take $R=n^q$:
\begin{equation*}
\ls(X_q)=\lim_{n\to\infty}\frac{\log((n+1)^q-n^q)}{\log n^q}.
\end{equation*}
We have
\begin{equation*}
\lim_{x\to\infty}\frac{(x+1)^q-x^q}{x^{q-1}}=
\lim_{x\to\infty}\frac{(1+1/x)^q-1}{1/x}=\lim_{y\to0^+}\frac{(1+y)^q-1}{y}=
q.
\end{equation*}
Therefore
\begin{equation*}
\ls(X_q)=\lim_{n\to\infty}\frac{\log(qn^{q-1})}{\log n^q}=\frac{q-1}q.
\end{equation*}
In order to compute $\us(X_q)$ we take $R=n^q+\eps$ and then the limit as
$\eps\to0$:
\begin{equation*}
\us(X_q)=\lim_{n\to\infty}\frac{\log((n+2)^q-(n+1)^q)}{\log n^q}.
\end{equation*}
Similarly as before, we get
\begin{equation*}
\us(X_q)=\lim_{n\to\infty}\frac{\log(q(n+1)^{q-1})}{\log n^q}=\frac{q-1}q.
\end{equation*}
Thus, $\sigma(X_q)=1-1/q$.
\ee

\begin{example}\label{2^n}
Now we consider the case of $a_n=2^n$. We get
\begin{equation*}
\ls(X)=\lim_{n\to\infty}\frac{\log(2^{n+1}-2^n)}{\log 2^n}=1
\end{equation*}
and similarly,
\begin{equation*}
\us(X)=\lim_{n\to\infty}\frac{\log(2^{n+2}-2^{n+1})}{\log 2^n}=1.
\end{equation*}
Thus, $\sigma(X)=1$.
\ee

\begin{example}\label{nlogn}
This example is a sparse space with sparsity order $\sigma = 0.$

Let $a_n = n \log n.$ To compute $\us(X),$ we take $R= n \log n +
\eps$ and then as before the limit as $\eps \to 0$. Note that
\begin{equation*}
\lim_{x\to\infty}\frac{(x+1)\log(x+1) - x \log x} {\log x} =
\lim_{x\to\infty}\frac{\log(1+1/x)}{1/x}=\lim_{y\to0^+}\frac{\log(1+y)}{y}= 1.
\end{equation*}
Therefore
\begin{equation*}
\us(X) = \lim_{n\to\infty}\frac{\log[(n+2)\log(n+2) - (n+1) \log
    (n+1)]} {\log(n\log n)} =
\lim_{n\to\infty}\frac{\log\log(n+1)}{\log (n \log n)} = 0.
\end{equation*}
Thus, $\sigma(X)=\us(X)=0.$
\ee

\begin{example}
Take $a_n=2^{b^n}$ for some $b>1$. We have
\begin{equation*}
\ls(X)=\lim_{n\to\infty}\frac{\log(2^{b^{n+1}}-2^{b^n})}{\log
  2^{b^n}}=\lim_{n\to\infty}\frac{b^{n+1}}{b^n}=b.
\end{equation*}
Similarly,
\begin{equation*}
\us(X)=\lim_{n\to\infty}\frac{\log(2^{b^{n+2}}-2^{b^{n+1}})}{\log
  2^{b^n}}=\lim_{n\to\infty}\frac{b^{n+2}}{b^n}=b^2.
\end{equation*}
\ee

It is easy to see that similar computations for the triple exponential
for $a_n$ give $\sigma(X)=\infty$.

Sparsity order at infinity is invariant under quasi-isometry. A map
$\ph: X \to Y$ between metric spaces is called a {\em
  $(C,A)$-quasi-isometric embedding} if for every $x, x' \in X$
$$C^{-1} d(x,x')-A \leq d(\ph(x),\ph(x')) \leq C d(x,x')+A.$$

If in addition $Y$ is contained in an $A$-neighborhood of $\ph(X)$,
i.e. for every $y\in Y$ there is $x\in X$ with $d(y, \ph(x)) \leq A,$
$\ph$ is called a {\em $(C,A)$-quasi-isometry,} in which case $X$ and
$Y$ are said to be {\em quasi-isometric.}

More generally, a map $\ph: X \to Y$ between metric spaces is called a
{\em coarse embedding} if the affine maps in the inequality are
replaced by arbitrary unbounded nondecreasing maps
$\rho_-,\rho_+:[0,\infty) \to [0,\infty)$ respectively. Such a $\ph$
    is a {\em coarse equivalence} if in addition $Y$ is contained in
    some $A$-neighborhood of $\ph(X)$.

Note that the quasi-isometric image of a sparse metric space is
sparse, and more generally the image by a coarse embedding of a sparse
metric space is sparse.

\begin{proposition}\label{qi-invar}
If metric spaces $X$ and $Y$ are sparse at infinity and
quasi-isometric and $\sigma(X)$ exists, then $\sigma(Y)$ exists and
$\sigma(Y)=\sigma(X)$.
\end{proposition}

\begin{proof}
Let $\ph: X \to Y$ be a $(C,A)$-quasi-isometry. Then for $d(x,x')$
sufficiently large, $\ph$ is $D:=C+1$ bi-Lipschitz:
$$D^{-1} d(x,x') \leq d(\ph(x),\ph(x')) \leq D d(x,x').$$

Then for sufficiently large $R,$
\begin{equation*}
s_Y (R/D) \leq D s_X (R)
\end{equation*}
so that
\begin{equation*}\varlimsup_{R\to \infty} s_Y (R/D) \leq
  \varlimsup_{R\to \infty} D s_X (R)
\end{equation*}
and so
\begin{equation*}
\us(Y) \leq \sigma(X).
\end{equation*}

Similarly, for $D'$ a large scale bi-Lipschitz constant for an inverse
quasi-isometry $\ph',$
\begin{equation*}
s_X (R/D') \leq D' s_Y (R)
\end{equation*}
so that
\begin{equation*}
\varliminf_{R\to \infty} s_X (R/D') \leq \varliminf_{R\to \infty} D'
s_Y (R)
\end{equation*}
and so
\begin{equation*}
\sigma(X) \leq \ls(Y).
\end{equation*}

So
\begin{equation*}
\ls(Y)=\us(Y)=\sigma(Y)=\sigma(X).
\end{equation*}
\end{proof}

In contrast to this result, note that coarse equivalence of sparse
metric spaces does not preserve sparsity order. For example, the
subspaces of the halfline $X_2=\{0\} \cup \{n^2 | n\in \mathbb{N}\}$
and $X_3=\{0\}\cup \{n^3 | n\in \mathbb{N}\}$ are coarsely equivalent,
for example with $\rho_-(x)=x$ and $\rho_+(x)=x^2$, while
$\sigma(X_2)=1/2$ and $\sigma(X_3)=2/3$.

We want to extend the sparsity order at infinity to spaces that are
quasi-isometric to spaces that are sparse at infinity. If a metric
space $X$ has a sparse subspace $X'$ such that $X$ is contained in
some $A$-neighborhood of $X'$, we call the subspace $X'$ a {\em sparse
  skeleton} of the space $X$. Clearly, any two sparse skeletons of a
metric space are quasi-isometric.

\begin{proposition}\label{qs}
{\rm (cf.~\cite{P}).}
The following are equivalent for a metric space $X:$
\item[(1)] $X$ is coarsely equivalent to a sparse space.
\item[(2)] $X$ is quasi-isometric to a sparse space.
\item[(3)] $X$ has a sparse skeleton.
\end{proposition}

\begin{proof}
Trivially (3)$\implies$(2)$\implies$(1). But also (1)$\implies$(3):
Let $\ph$ be a coarse equivalence from a sparse space $Y$ to $X.$ Then
$X':=\ph(Y)$ is a sparse subspace of $X$ and $X$ is contained in some
$A$-neighborhood of $X'.$
\end{proof}

\begin{definition}
We call a metric space satisfying the conditions of
Proposition~\ref{qs} {\em quasi-sparse}.
\end{definition}

Note that an unbounded subspace of a sparse metric space is sparse,
and similarly an unbounded subspace of a quasi-sparse space is
quasi-sparse.

\begin{remark}
If $X$ is quasi-sparse, in light of Proposition~\ref{qi-invar} we can
define its lower and upper sparsity order at infinity as that of any
of its sparse skeletons, and similarly for its sparsity order at
infinity.
\end{remark}

The following definition appears in \cite{BCL}, where it is called sparse
rather than finitely sparse and is used in connection with the
rigidity of Roe algebras; see also \cite{BFV} and \cite{CMS}.

\begin{definition}
Let $(X,d)$ be a metric space. It is {\em finitely sparse} if there
exists a partition $X = \bigsqcup_{n=1}^\infty X_n,\ $ such that the
cardinality of each $X_n$ is finite and $d(X_n, X_m) \rightarrow
\infty$ as $n+m \rightarrow \infty.$
\end{definition}

Clearly a finitely sparse metric space for which the $X_n$ can be
taken to be singletons must be sparse, and more generally one for
which they can be taken to be uniformly bounded in diameter must be
quasi-sparse.

In the other direction, a quasi-sparse locally finite metric space is
finitely sparse. (Recall that a locally finite metric space is one in
which every ball has finite cardinality.)

\begin{example}\label{2shift}
Consider the one-sided full shift on two symbols
\begin{equation*}
\Sigma_2 = \{0,1\}^\N.
\end{equation*}
We can view a point $x\in \Sigma_2$ as the subspace of the metric
space $\N$ determined by the locations of the 1's in $x.$ If $$\sigma:
\Sigma_2 \to \Sigma_2$$ is the left shift
\begin{equation*}
\sigma(x_1 x_2 x_3 \ldots) = (x_2 x_3 x_4 \ldots),
\end{equation*}
then $x$ is sparse (respectively quasi-sparse, finitely sparse) if and
only if every point in its $\sigma$-orbit is also sparse (respectively
quasi-sparse, finitely sparse).
\ee

\begin{example}\label{sparseshift}
Let $S \subset \Sigma_2$ be the subshift consisting of those sequences
in which the number of consecutive 0's preceding successive 1's is
strictly increasing. Then every $s \in S$ with infinitely many 1's is
sparse.
\ee

\begin{example}\label{cantor}
Consider the Cantor substitution $0 \longrightarrow 000, \ 1
\longrightarrow 101$, starting with 1. (Here the initial segment of
length $3^n$ shows which triadic subintervals of the unit interval are
present in the $n$-th stage of the construction of the Cantor set.)
The set of integers corresponding to the resulting sequence
$$101000101000000000101000101 \ldots$$
is finitely sparse but not quasi-sparse.
\ee

\begin{example}\label{sierpinski}
Consider the two-dimensional Sierpi\'{n}ski substitution taking 0 to a
3 by 3 block of 0's and 1 to a 3 by 3 block with 1's in the corners
and 0's elsewhere, starting with 1. The set of points in the positive
quadrant of the integer lattice in the plane corresponding to this
substitution is similarly finitely sparse but not quasi-sparse.
\ee

\begin{example}\label{primes}
Let $P$ be the set of prime numbers. Then Zhang's celebrated result on
bounded gaps between primes \cite{Z} shows that $P$ is not sparse. On the
other hand, the fact that there are successive primes with arbitrarily
long gaps between them implies that $P$ is finitely sparse. But as the
following proposition shows, it follows easily from results of Maynard
on the heels of Zhang that $P$ is not quasi-sparse.
\ee

\begin{proposition}\label{Pnotqs}
The primes are not quasi-sparse.
\end{proposition}

\begin{proof}
If $P$ were quasi-sparse, there would be a sparse skeleton
$\{p_{s_i}\}$ and an $a>0$ such that $p_{s_{i+1}} - p_{s_i}
\longrightarrow \infty,$ and for each $p \in P$ there would exist $s_i
=: s(p)$ with $|p - p_{s_i}| \leq a.$

Fix $m > 2a.$ Then by Maynard \cite{M}, $\liminf (p_{n+m} - p_n) <
\infty$, so there exists a sequence $\{n_i\}$ such that $(p_{n_i+m} -
p_{n_i}) \longrightarrow M < \infty.$

But $p_{n_i+m} - p_{n_i} \geq m > 2a,$ so $s'_i :=s(p_{n_i+m}) > s_i
:=s(p_{n_i}).$

So $p_{n_i+m} - p_{n_i} = p_{n_i+m} - p_{s'_i } + p_{s'_i } - p_{s_i }
+ p_{s_i } - p_{n_i} \geq p_{s'_i } - p_{s_i } - 2a \longrightarrow
\infty.$
\end{proof}

Gromov \cite{G} introduced the notion of asymptotic dimension
(see also~\cite{BD})
originally as a geometric invariant for finitely generated groups.

\begin{definition}
Let $(X,d)$ be a metric space. $X$ has {\em asymptotic dimension zero}
if for every $r > 0$, there is a partition of $X$ into uniformly
bounded $r$-disjoint sets.
\end{definition}

Here a family of sets is $r$-disjoint if every distinct pair of sets
is at distance larger than $r$.

\begin{proposition}\label{qs-ad0}
{\rm (cf.~\cite{P}).}
If a metric space $X$ is quasi-sparse, then it has asymptotic
dimension zero.
\end{proposition}

\begin{proof}
If $X$ is sparse, given $r$ a partition consisting of all singletons
except for a single sufficiently large ball works. Since a
quasi-sparse space is coarsely equivalent to a sparse space and
asymptotic dimension is a coarse invariant, quasi-sparse spaces also
have asymptotic dimension zero.
\end{proof}

On the other hand, a metric space of asymptotic dimension zero need
not be quasi-sparse, since in particular, the product of sparse spaces
is not quasi-sparse, while it has asymptotic dimension zero as the
product of spaces with asymptotic dimension zero \cite{BD}:

\begin{proposition}\label{prod}
If $X$ and $Y$ are sparse spaces, then $X\times Y$ is not
quasi-sparse.
\end{proposition}

\begin{proof}
Suppose $X\times Y$ is quasi-sparse. Then it has a skeleton $Z.$ This
$Z$ is $A$-dense in $X\times Y$ for some $A>0.$ Choose $y_0\in Y.$ If
$d(y_0,y)$ is sufficiently large, then $Z\cap (X\times \{y\})$ is
$A$-dense in $X\times \{y\}.$ There exist $x_1, x_2\in X$ (independent
of $y$) such that $d(x_1,x_2)>2A.$ Then there are two distinct
elements of $Z\cap(X\times \{y\})$ whose distance from one another is
at most $d(x_1,x_2)+2A.$ This happens for every $y$ for which
$d(y_0,y)$ is sufficiently large, so $Z$ is not sparse, a
contradiction.
\end{proof}

\section{Sparsity and tryposity at infinity}\label{tryp}

In \cite{GM2}, the authors introduced the {\em convexity gap} in metric
spaces, and used its local behavior to study curvature. In this
section, we relate sparsity to the large scale behavior of this
nonconvexity measure in metric spaces. The subsequent sections of this
paper are independent of this section.

Let $(X,d)$ be a metric space, $x_0 \in X,$ and $\oB(x,R)$ be the
closed ball in $X$, centered at $x\in X$ and of radius $R$. For
$x,y\in X$ we define the \emph{hyperdistance} between $x$ and $y$ as
\begin{equation*}
hd(x,y)=2\;\inf\{r: \oB(x,r)\cap \oB(y,r)\neq\emptyset\}.
\end{equation*}

Next, we define the
\emph{leipodistance} between $x$ and $y$ by
\begin{equation*}
ld(x,y)=hd(x,y)-d(x,y).
\end{equation*}
Observe that $d(x,y)\le hd(x,y)\le2 d(x,y)$, so $0\le ld(x,y)\le
d(x,y)$. The leipodistance between a pair of points is a measure of
the degree of their failure to have a midpoint.

Now we define the \emph{convexity gap} (or simply \emph{gap}) of $X$
as
\begin{equation*}
\gamma(X)=\sup\{ld(x,y):x,y\in X\}.
\end{equation*}

In \cite{GM2}, the authors studied how the gap of a small ball around a point on a
curve shrank as its radius shrank, and related this {\em tryposity} to
curvature. In contrast, here we will look at how the gap of a large
ball grows as its radius grows, and relate the resulting notion of
tryposity at infinity to sparsity.

Define the \emph{lower and upper tryposity order at infinity} ($\lt$
and $\ut$) as
\begin{equation*}
\lt(X)=\varliminf_{R\to\infty}\frac{\log \gamma(\oB(x_0,R))}{\log R}, \ \ \ \
\ut(X)=\varlimsup_{R\to\infty}\frac{\log \gamma(\oB(x_0,R))}{\log R}.
\end{equation*}
If $\lt(X)=\ut(X)$, we will speak simply of the \emph{tryposity order
  at infinity} and denote it $\tau(X)$. These do not depend on the
choice of $x_0$.

We revisit the first two examples of the previous section, where
$X=\{0\}\cup\{a_n:n=1,2,3,\dots\}\subset\R$.

\begin{example}\label{trypq}
Fix $q>1$. Then $X_q$ takes $a_n=n^q$. In order to compute $\ut(X_q)$
we take $R=(n+1)^q$:
\begin{equation*}
\ut(X)=\varlimsup_{R\to\infty}\frac{\log (\oB(x_0,R)}{\log R}
=\lim_{n\to\infty}\frac{\log((n+1)^q-n^q)}{\log (n+1)^q}.
\end{equation*}
We use that
\begin{equation*}
\lim_{x\to\infty}\frac{(x+1)^q-x^q}{(x+1)^{q-1}}=
\lim_{x\to\infty}\frac{(x+1)^q-x^q}{x^{q-1}}=q
\end{equation*}
as bef

Therefore
\begin{equation*}
\ut(X_q) =\lim_{n\to\infty}\frac{\log((n+1)^q-n^q)}{\log (n+1)^q}
=\frac{q-1}q.
\end{equation*}
In order to compute $\lt(X_q)$ we take $R=(n+1)^q-\eps$ and then the
limit as $\eps\to0$:
\begin{equation*}
\lt(X_q)=\lim_{n\to\infty}\frac{\log(n^q-(n-1)^q)}{\log (n+1)^q}.
\end{equation*}
Similarly as before, we get
\begin{equation*}
\lt(X_q)=\lim_{n\to\infty}\frac{\log(q(n-1)^{q-1})}{\log
  (n+1)^q}=\frac{q-1}q.
\end{equation*}
Thus, $\tau(X_q)=1-1/q$.
\ee

\begin{example}\label{trypexp}
Now we consider the case of $a_n=2^n$. We get
\begin{equation*}
\ut(X)=\lim_{n\to\infty}\frac{\log(2^{n+1}-2^n)}{\log 2^{n+1}}=
\lim_{n\to\infty}\frac{n}{n+1}=1
\end{equation*}
and similarly,
\begin{equation*}
\lt(X)=\lim_{n\to\infty}\frac{\log(2^n-2^{n-1})}{\log 2^{n+1}}=
\lim_{n\to\infty}\frac{n-1}{n+1}=1.
\end{equation*}
Thus, $\tau(X)=1$.
\ee

\begin{remark}\label{trypnbhd}
If a metric space $X$ is contained in some $A-$neighborhood of a
subspace $X'$, then $\ut(X')=\ut(X)$ and $\lt(X')=\lt(X)$ so that
$\tau(X')=\tau(X)$ if they exist.
\end{remark}

In general, the order of tryposity at infinity is not invariant under
quasi-isometry:

\begin{example}\label{mountainrange}
(The Mountain Range) Let $M$ be the subset of the plane consisting of
  the origin and segments from $(2^n,0)$ to $((3/2) 2^n, 2^n)$ and
  from $((3/2)2^n,2^n)$ to $(2^{n+1},0)$ for $n \in \Z.$ Then
  $\tau(M)=1.$ But $M$ is quasi-isometric to the nonnegative
  horizontal axis since projection onto the horizontal coordinate is a
  bi-Lipschitz map from $M$ onto the nonnegative reals, and the
  halfline has $\tau = 0.$
\end{example}

The following modification of Example~\ref{mountainrange} shows that
even among sparse spaces the tryposity order at infinity is not a
quasi-isometry invariant.

\begin{example}\label{trypnotq-iinvar}
Consider the subspace $M'$ of $M$ consisting of the origin and the
endpoints of the segments of $M$ and $2^{n/2}$ equally spaced points
in the interior of the segment starting at $2^n$ and the segment
ending at $2^{n+1}$ on the $x$-axis. Then $M'$ is sparse, with
$\sigma(M')=1/2,$ and its tryposity order at infinity is
$\tau(M')=\tau(M)=1,$ but its image in the $x$-axis by the
bi-Lipschitz projection map has $\tau = 1/2.$
\end{example}

\begin{proposition}\label{tryp/sparsity}
If $X$ is quasi-sparse of order $0 < \sigma \leq 1$ then the upper
tryposity order at infinity $\ut(X) \geq \sigma(X).$
\end{proposition}

\begin{proof}
By Remark ~\ref{trypnbhd}, it suffices to prove the proposition for a
sparse skeleton of a given quasi-sparse space. So we assume $X$ is
sparse of order $\sigma.$

For all $x' \in X \setminus B(x_0,R+s(R)),$ and for all $x \in X
\setminus B(x_0,R),$ we have $d(x,x') \geq s(R).$ So for all $x',x''
\in X \setminus B(x_0,R+s(R)),$ we have $ld(x',x'') \geq s(R).$ Thus
$\gamma(\oB(x_0,R+s(R))) \geq s(R),$ and
\begin{equation*}
\frac{\log{\gamma(\oB(x_0,R+s(R)))}}{\log(R+s(R)} \geq \frac{\log
  s(R)}{\log(R+s(R)} = \frac{\log s(R)}{\log R}  \frac{\log
  R}{\log(R+s(R)}.
\end{equation*}
So
\begin{equation*}
\begin{split}
\ut(X) = \varlimsup \frac{\log{\gamma(\oB(x_0,R))}}{\log(R} &\geq
\varlimsup \frac{\log{\gamma(\oB(x_0,R+s(R)))}}{\log(R+s(R)} \\ &\geq
\varlimsup \frac{\log s(R)}{\log R} \frac{\log R}{\log (R+s(R)} =
\sigma.
\end{split}
\end{equation*}
\end{proof}

The following modification of Example~\ref{trypq} shows that a sparse
space's tryposity order at infinity may exceed its sparsity order.

\begin{example}
Let $X = X_3 \cup \{n^3 + n : n=1,2,3,\dots\}.$
Then $\tau(X) = \tau(X_3) = 2/3$ but $\sigma(X) = 1/3.$
\end{example}

In fact, a space may not even be quasi-sparse but still have tryposity
order at infinity 1, as this modification of Example~\ref{trypexp}
shows:

\begin{example}
Let $X=\{0\} \cup \bigcup_{k=1}^\infty [2^{2k}, 2^{2k+1}].$ Then
$\tau(X) =1$ as in Example ~\ref{trypexp}, but $X$ is not
quasi-sparse.
\end{example}

\section{Sparse ends}\label{ends}

We want to define analogues of topological ends in metric spaces that
are not connected or even coarsely connected, like sparse spaces.

Let $(X,d)$ be a metric space with a base point $x_0$. For
$\alpha\in[0,1]$ and $K>0$ we define the relation of being
\emph{$(\alpha,K)$-close} (relative to the base point $x_0)$. Namely,
$x,y\in X$ are in this relation if
\begin{equation*}
d(x,y)\le K[\min\{d(x,x_0),d(y,x_0)\}]^\alpha + K.
\end{equation*}
Let $S$ be a subset of $X.$ We say $x,y$ are
\emph{$(\alpha,K)$-chained} (in $S$) if there are points
$x_1,x_2,\dots,x_n$ (for some $n>1$) in $S$ such that $x_1=x$,
$x_n=y$, and $x_i,x_{i+1}$ are $(\alpha,K)$-close for
$i=1,2,\dots,{n-1}$. This is of course an equivalence relation.

We call the equivalence classes the \emph{$(\alpha,K)$-components} of
$S$. If $S$ has just one \akcomp, we call it \emph{\akconn}. If $S$ is
\akconn \ for some $K>0$, we call it \emph{\aconn}.

\begin{remark} (i) The \akcomp s of $S$ are the maximal \akconn
  \ subsets of $S$.
\item[(ii)] If $S$ is \akconn, then it is \akpconn \ for all $K' \geq K.$
\item[(iii)] If $S=A \cup B$ is \akconn, with $A,B\neq \emptyset,$
  then there are \ak-close \ $x\in A, y\in B.$
\item[(iv)] If $A$ and $B$ are \akconn \ subsets of $S$ and some
  $x\in A$ and $y\in B$ are \ak-close, then $A\cup B$ is \akconn.
\end{remark}

\begin{remark}\label{l1}
Let $X = \{x_0, x_1, x_2, \ldots \}$ be the set of terms of an
increasing sequence in the halfline $[0,\infty)$ going to infinity.
  Then X is \akconn \ if and only if all pairs of consecutive elements
  $x_i,x_{i+1}\in X$ are ($\alpha,K$)-close.
\end{remark}

If $A\subset X,$ we write $d(A,x_0)$ for $\inf\{d(x,x_0) | x\in A\}.$

\begin{lemma}
Let $R>0.$ A subset $S$ of a pointed metric space $(X,x_0,d)$ is
\aconn \ if and only if $S \setminus B(x_0,R)$ is \aconn.

More specifically: if $S$ is \akconn, then $S \setminus B(x_0,R)$ is
\akpconn, where $K'= 2K(R^\alpha +1) +2R,$ and on the other hand if $S
\setminus B(x_0,R)$ is \akconn, then $S$ is \akpconn, where
$K'=\max\{K, 2d(S \setminus B(x_0,R), x_0)+1\}.$
\end{lemma}

\begin{proof}
Assume first that $S$ is \akconn. We show that every $x,y \in
S\setminus B(x_0,R)$ are \akpch \ in $S\setminus B(x_0,R),$ where $K'
= 2K(R^\alpha +1) +2R.$

Note that it suffices to show this when each $z$ other than $x$ and
$y$ in the \ak-chain connecting them is in $B(x_0,R).$

Let this \ak-chain be $x,x',\ldots,y',y.$ Then
\begin{equation*}
\begin{split}
d(x,y) &\leq d(x,x') + d(x',y') + d(y',y)\\
	& \leq Kd(x',x_0)^\alpha +K + 2R + Kd(y',x_0)^\alpha +K\\
	& \leq 2KR^\alpha + 2K +2R\\
	& \leq K'
 \end{split}
\end{equation*}
so that $x$ and $y$ are \akp-close.

Now instead assume that $S \setminus B(x_0,R)$ is \akconn.

Let $M=d(S \setminus B(x_0,R), x_0),$ so that $M \geq R,$ and pick
$y_0 \in S \setminus B(x_0,R)$ with $d(x_0,y_0) \leq M+1.$ Set
$K'=\max\{K, 2M+1\}.$ Consider $x,y \in S.$

If $x,y \in S \setminus B(x_0,R),$ then they are \akch \ (in $S
\setminus B(x_0,R)$) and thus \akpch.

If $x,y \in S \cap B(x_0,R),$ then they are $(\alpha,2R)$-close and
thus $(\alpha,K')$-close.

If $x \in S \cap B(x_0,R)$ and $y \in S \setminus B(x_0,R),$ then they
are \akpch \ (in $S$) since $d(x,y_0) \leq 2M+1$ so that $x$ is
$(\alpha,K')$-close to $y_0.$
\end{proof}

Note that the first half of the proof shows that if every $x,y \in
S\setminus B(x_0,R)$ are \akch \ in $S,$ then they are \akpch \ in
$S\setminus B(x_0,R)$ (so that $S\setminus B(x_0,R)$ is \akpconn).

\begin{proposition}\label{aconninvariance}
Quasi-isometries preserve \aconn ness. That is, if $S$ is an \aconn
\ subset of a pointed metric space $(X,x_0,d)$ and $\ph: (X,x_0,d) \to
(X',x'_0,d')$ is a quasi-isometry to another pointed metric space,
then $S'=\ph(S)$ is \aconn. More specifically, if $S$ is \akconn, then
$S'$ is \akpconn, where $K'=\max\{C(K+1), 2^\alpha C^{1+\alpha} K\}$
for a constant $C$ associated with the quasi-isometry.
\end{proposition}

\begin{proof}
Since $\ph$ is a quasi-isometry, there is $C>1$ such that for all $x,y
\in S,$ we have
\begin{equation*}
d'(\ph(x),\ph(y)) \leq Cd(x,y)+C
\end{equation*}
and
\begin{equation*}d(x,y) \leq Cd'(\ph(x),\ph(y))+C.
\end{equation*}
By the Lemma, it suffices to show that $S' \setminus B(x'_0,2C^2 +C)$
is \aconn. Again by the Lemma, there is some $K > 0$ so that $S
\setminus B(x_0,2C)$ is \akconn.

Let $x',y' \in S' \setminus B(x'_0,2C^2 +C)$ and let $K'=\max\{C(K+1),
2^\alpha C^{1+\alpha} K\}.$ By the remark after the Lemma, it suffices
to show that $x',y'$ are \akpch \ in $S'$ (since then they will also
be \akpch \ in $S' \setminus B(x'_0,2C^2 +C)$).

Note first that $\ph(B(x_0,2C)) \subset B(x'_0,2C^2 +C).$ So there are
$x,y \in S \setminus B(x_0,2C)$ with $\ph(x)=x', \ph(y)=y'$ and so
also an \ak-chain $x=x_1,\ldots, x_n=y$ in $S \setminus B(x_0,2C).$
Let $x'_i = \ph(x_i)$ for each $i.$ So for each $i,$ $x'_i \notin
B(x'_0,1)$ since $x_i \notin B(x_0,2C).$ Then for each $i=1,\ldots,
n-1$ we have
\begin{equation*}
\begin{split}
d'(x'_i,x'_{i+1}) &\leq Cd(x_i,x_{i+1}) + C\\
	& \leq C(K[\min\{d(x_i,x_0),d(x_{i+1},x_0)\}]^\alpha +K) + C\\
	& = CK[\min\{d(x_i,x_0),d(x_{i+1},x_0)\}]^\alpha +C(K+1)   \\
	& \leq CK[C\min\{d'(x'_i,x'_0),d'(x'_{i+1},x'_0)\} +C]^\alpha +C(K+1)   \\
	& = C^{1+\alpha}K[\min\{d'(x'_i,x'_0),d'(x'_{i+1},x'_0)\} +1]^\alpha +C(K+1)   \\
	& \leq C^{1+\alpha}K[2\min\{d'(x'_i,x'_0),d'(x'_{i+1},x'_0)\}]^\alpha +C(K+1)   \\
	& =2^\alpha C^{1+\alpha}K[\min\{d'(x'_i,x'_0),d'(x'_{i+1},x'_0)\}]^\alpha +C(K+1)   \\
	& \leq K'[\min\{d'(x'_i,x'_0),d'(x'_{i+1},x'_0)\}]^\alpha +K'
 \end{split}
\end{equation*}
So $x'$ and $y'$ are \akpch\ in $S'$ as we wanted.
\end{proof}

\begin{example}\label{exlip}
Although by Proposition~\ref{aconninvariance} a bi-Lipschitz image of
an \aconn \ set is \aconn, this example shows that if $X$ is
$\alpha$-connected and $f:X\to Y$ is Lipschitz (even with Lipschitz
constant 1), then it may happen that $f(X)$ is not $\alpha$-connected
(with the same $\alpha$).

Set $X=\{i^2:i=0,1,2,3,\dots\}$ and $\alpha=1/2$. Let us construct $f$
and $Y=f(X)$ by induction. Our set $Y$ will be of the form
$\{y_n:n=0,1,2,3,\dots\}$ with $y_0\le y_1\le y_2\le y_3\dots$. We
start with $y_0=0$. Suppose that $y_i$ is defined and $f(i^2)=y_i$ for
$i=0,1,2,\dots,n$. If $k$ is sufficiently large, we have
\begin{equation*}
2k-1= d\big(k^2,(k-1)^2\big)>n\big(d(y_0,y_n)\big)^\alpha+n =n
\sqrt{y_n} +n.
\end{equation*}
We fix such $k>n+1$ and set $y_i=f(i^2)=y_n$ for
$i=n+1,n+2,n+3,\dots,k-1$ and $y_k=f(k^2)=y_n+k^2-(k-1)^2$. Then
$y_{k-1}$ and $y_k$ are not ($\alpha,n$)-close. By Remark~\ref{l1},
$f(X)$ is not ($\alpha,n$)-connected. Since this happens for
infinitely many $n$, $f(X)$ is not $\alpha$-connected.
\ee

\begin{corollary}
Let $S$ be a subset of a metric space $(X,d)$ and $x_0,y_0 \in X.$
Then if $S$ is \aconn \ relative to $x_0$ then it is \aconn \ relative
to $y_0.$

More specifically, if $S$ is \akconn \ relative to $x_0$, then it is
\akpconn \ relative to $y_0,$ where $K'=\max\{C(K+1), 2^\alpha
C^{1+\alpha} K\}$ and $C = d(x_0,y_0).$
\end{corollary}

\begin{proof}
Let $\ph:(X,x_0) \to (X,y_0)$ be the quasi-isometry taking $x_0$ to
$y_0$ and leaving every other point fixed.
\end{proof}

Given an unbounded subset $S$ of a pointed metric space $(X,x_0,d),$
let us look at the set of unbounded \akcomp s of $S\setminus
B(x_0,R),$ provided with the discrete topology. For $R<R'$, each
unbounded \akcomp \ for $R'$ is contained in some unbounded \akcomp
\ for $R$. However, ``splitting'' can occur, so the number of
unbounded \akcomp s for $R'$ can be larger than for $R$, but it cannot
be smaller.

Let the {\em real \akray} be the (barely) \akconn \ sequence of reals
$\Rak = \{(a_n)_{n=0}^\infty | a_0=0, a_n = a_{n-1} + Ka_n^\alpha + K,
n>0 \} \subset \R_{\ge 0}.$

The image in $S \subset X$ of $\Rak$ under a quasi-isometry is
\akpconn \ for $K'$ as in the proposition. We define an {\em \akqray}
in $S \subset X$ as an \akconn \ sequence in $S$ that is the image of
some real \akpray \ $\Rakp$ under a quasi-isometry, and we say that
the \akqray \ is {\em based} at the image of 0.

Note that if $(x_n)$ is an \akqray, then for any $R >0$, $(x_n)
\setminus B(x_0,R)$ has exactly one unbounded \akcomp.

If $(x_n)$ and $(y_n)$ are two \akqray s in $S$, set $(x_n)\sim (y_n)$
if for all $R >0$, the unbounded \akcomp \ of $(x_n) \setminus
B(x_0,R)$ and of $(y_n) \setminus B(x_0,R)$ lie in the same
(unbounded) \akcomp \ of $S \setminus B(x_0,R)$.

We define an \emph{$(\alpha,K)$-end} $A$ as an equivalence class of
\akqray s.

A basic open set of \akend s is a set of equivalence classes of
representative sequences $(x_n)$ such that for some $R>0$, the
unbounded \akcomp \ of each $(x_n) \setminus B(x_0,R)$ lies in the
same (unbounded) \akcomp \ of $S \setminus B(x_0,R)$.

For example, if $X$ is the binary tree with root $x_0$ viewed as a
metric space, $S$ is the vertex set of $X$, and we take $\alpha = 0$,
the set of \akend s of $S$ is empty for $K<1$ while for $K \ge 1$ it
corresponds to the set of infinite paths in $X$ topologized as a
Cantor set, i.e. the usual space of topological ends of $X$.

Clearly \akend s aren't invariant under quasi-isometries, as \akconn
ness isn't. But we use them to define \aend s, which will be
invariant.

Now we keep $\alpha$ constant, but vary $K$. Here we do not have
monotonicity. As $K$ grows, two \akend s can be glued together, but
also new \akend s can be created. Thus, we will take a direct limit
over $K$ to define the space of \aend s.

First note that if $K<K'$, then every \akend \ is contained in some
\akpend. Thus there is a well defined map from \akend s to \akpend s.

Fix $\alpha.$ For each $K>0,$ let $\calek$ be the set of all \akend s.
We define an equivalence relation $\sim$ on the disjoint union
$\bigsqcup_{K>0} \calek$: if $E \in \calek$ and $E' \in \calekp$ then
$E \sim E'$ if there is some $L \geq K,K'$ and some $F \in {\mathcal
  E}_L$ such that $E$ and $E'$ are contained in $F$.

An equivalence class of this relation will be called an
\emph{$\alpha$-end}. That is, the set of \aend s of the subset $S$ of
the pointed metric space $(X,x_0,d)$ is
\begin{equation*}
\caleinf = \caleinf(S) = \bigsqcup_{K>0} \calek / \sim.
\end{equation*}

So every $e \in \caleinf$ is an equivalence class $e=[E]$, where $E
\in \calek$ for some $K$.

We endow the set of \aend s with the direct limit topology, so that a
set of \aend s is open when for each $K$ the corresponding set of \ak
\ representatives is.

\begin{proposition}\label{aendinvariance}
A quasi-isometry induces a homeomorphism of the spaces of \aend s.

That is, if $\alpha \geq 0$ and $S$ is a subset of a pointed metric
space $(X,x_0,d)$ with \aend s $\caleinf(S)$ and $\ph: (X,x_0,d) \to
(X',x'_0,d')$ is a quasi-isometry to another pointed metric space,
then $\caleinf(S)$ is homeomorphic to $\caleinf(S')$, where
$S'=\ph(S)$.
\end{proposition}

\begin{proof}
We first show that $\ph$ induces a bijection of the set of \aend s. By
symmetry, it suffices to show that $\ph$ induces a map from \aend s of
$S$ to \aend s of $S'$ and then check that this map is one to one.

If $e$ is an \aend \ of $S$, then $e=[E]$, where $E \in \calek(S)$ for
some $K$. Then it follows from Proposition \ref{aconninvariance} that
$\ph$ takes $E$ to some $E' \in \calekp(S')$, where $K'$ is as given
in the proposition. So the induced map $\ph_*$ on \aend s takes $e \in
\caleinf(S)$ to $e' \in \caleinf(S')$, where $e' = [E']$.

Now suppose that $e_1$ and $e_2$ are \aend s of $S$ and
$\ph_*(e_1)=\ph_*(e_2)=e'$. Let $\psi: (X',x'_0,d') \to (X,x_0,d)$ be
an inverse quasi-isometry of $\ph$ with $\psi(S') \subset S$ and
$\psi(S')$ $K_0$-dense in $S$, and let $e= \psi_*(e')$. If $e_1 =
[E_{K_1}]$ and $e_2 = [E_{K_2}]$, and we take $E_K$ with $e = [E_K]$
and $K \geq \max(K_1,K_2)$, then $E_{K_1}$ and $E_{K_2}$ are contained
in $E_K$ so we have that $e_1 = e_2$. So $\ph_*$ is a bijection on the
set of \aend s.

It remains to prove that $\ph_*$ is a homeomorphism. By symmetry, it
suffices to show that $\ph_*$ is continuous. We show that $\ph_*$ is
continuous at every \aend \ $e$ of $S$.

Let $e=[E]$, where $E \in \calek(S)$ for some $K$. Let $e'=\ph_*(e)
=[E']$, where $E' \in \calekp(S')$ for $K' \geq K'_0$, where $K'_0$ is
the constant given in Proposition \ref{aconninvariance}.

A basic neighborhood $V$ of the \aend\ $e'$ corresponds to the
unbounded \akpconn\ component of $S' \setminus B(x'_0, R')$ containing
$e'$ for some $R'$. We find a basic neighborhood $U$ of $e$ mapped by
$\ph$ into $V$. Let $R=CR'+C$, where $C$ is the constant in
Proposition \ref{aconninvariance}. Then $\ph(X \setminus B(x_0, R))
\subset X' \setminus B(x'_0, R')$ since $\ph(B(x_0, R)) \supset
B(x'_0, R')$. So $\ph(S \setminus B(x_0, R)) \subset S' \setminus
B(x'_0, R')$. Let $U$ be the basic neighborhood of $e$ corresponding
to the unbounded \akconn\ component of $S \setminus B(x_0, R)$
containing $e$. Then $\ph(U)$ is an \akpconn\ subset of $S' \setminus
B(x'_0, R')$, and it contains $e'$, so it is contained in $V$.
\end{proof}

\begin{corollary}
The space of \aend s is invariant under change of basepoint.
\end{corollary}

\begin{remark}
In light of Proposition \ref{aendinvariance}, for $\alpha \ge 0$ we
can define the space of \aend s of a quasi-sparse space as the space
of \aend s of any of its sparse skeletons.
\end{remark}

\begin{remark}
Our $0$-connectedness is just coarse connectedness (e.g.~\cite{CH}), i.e.
Gromov's \cite{G} long-range connectedness. Our 0-ends coincide with the
coarse ends of Alvarez Lopez and Candel \cite{AC}.
\end{remark}

\section{Examples}\label{examples}

Let us provide some examples.

\begin{example}\label{ex1}
Let $X$ be a subset of $[0,\infty)$, consisting of points $n^q$, where
  $q\ge1$ is a constant, and $n=0,1,2,\dots$. Set $x_0=0$ as the base
  point. Fix $\alpha\in[0,1]$ and $K>0$. If $n>0$, then $n^q$ and
  $(n+1)^q$ are $(\alpha,K)$-close if and only if
\begin{equation}\label{e1}
(n+1)^q-n^q\le Kn^{q\alpha}+K.
\end{equation}
We have
\begin{equation*}
\begin{split}
\lim_{x\to\infty}\frac{(x+1)^q-x^q}{x^{q-1}}&=
\lim_{x\to\infty}\frac{(1+\frac1x)^q-1}{\frac1x}=
\lim_{y\to0}\frac{(1+y)^q-1}{y}\\&=
\frac{d}{dy}\big((1+y)^q\big)\Big|_{y=0}=
q(1+y)^{q-1}\Big|_{y=0}=q.
\end{split}
\end{equation*}

Therefore,~\eqref{e1} holds for large $K$ and $n$ if and only if
$q-1\le q\alpha$, that is, if and only if $\alpha\ge(q-1)/q$. We see
easily that $X$ has no $\alpha$-ends if $\alpha<(q-1)/q$ and has one
$\alpha$-end if $\alpha\ge(q-1)/q$.

In particular, if we take $q=2$ we see that the set of squares
$X=\{n^2:n=0,1,2,3,\dots\}$ has one $\alpha$-end if $\alpha \ge 1/2$
and no $\alpha$-ends if $\alpha < 1/2$. \ee

\begin{remark}
White and Willett (\cite{WW}, Example 3.3) describe a Cartan subalgebra of
the uniform Roe algebra of the set of squares (or any sparse set of
nonnegative integers).

\end{remark}
\begin{example}\label{ex2}
Let us modify the preceding example. Namely, instead of the half-line
(which now can be viewed as one ray starting at the origin in $\R^2$)
we take $M$ rays starting at the origin. We set the origin as the base
point $x_0$. On the $m$th ray we take as elements of $X$ points whose
distance from $x_0$ is $n^{q_m}$, where $n=0,1,2,\dots$. Here
$q_1,\dots,q_M$ are constants larger than or equal to 1.

If $K$ and $\alpha$ are fixed and $R$ is large enough, no points from
different rays outside $B(x_0,R)$ are $(\alpha,K)$-close. Therefore,
by the same arguments as in Example~\ref{ex1}, we see that the number
of $\alpha$-ends of $X$ is equal to the number of those $m$ for which
$(q_m-1)/q_m\le\alpha$.
\ee

\begin{example}\label{ex3}
We can modify our example even further. Namely, instead of finitely
many rays, we can take infinitely many, such that in the intersection
of their union with the unit circle every point is isolated. Then
clearly the number of $\alpha$-ends of $X$ is at least the number of
those $m$ for which $(q_m-1)/q_m\le\alpha$.
\ee

\begin{example}\label{ex4}
Take points $z_m=(\cos(1/m),\sin(1/m))$ on the unit circle in $\R^2$,
where $m=1,2,3,\dots$. Define $X=\{n^{q_m}z_m:m,n=1,2,3,\dots\} \cup
\{(0,0)\}$ for some sequence $(q_m), q_m >1$. This is the situation
from the preceding example, and we see $X$ has at least zero $0$-ends,
which gives us no information. When we look along the rays, we cannot
get a better estimate. Yet we can define the sequence $(q_m)$ in such
a way that $X$ will have at least one $0$-end.

We do it in the following way. We partition the sequence $(1/m)$ into
infinitely many subsequences. Then we take a bijection from the set of
subsequences onto the set of natural numbers larger than 1. For each
subsequence $(y_k)$ we take numbers $a_k$ such that $2^{a_k}$
converges to the natural number corresponding to the subsequence
$(y_k)$. In such a way the set $\{(n,0):n=2,3,4,\dots\}$ belongs to
the closure of $X$. Therefore, $X$ has at least one $0$-end, which we
see when going along this set to infinity. This is a subset of the
horizontal ray, but no points of this ray belong to $X$, so we cannot
use Example~\ref{ex3} to see this end.
\ee

\begin{example}\label{ex5}
Denote the space from Example~\ref{ex1} by $X_q$. Now take the set
$X=X_{q_1}\times X_{q_2}$ for some $q_1,q_2\ge1$. If
$\alpha\ge(q_i-1)/q_i$ for $i=1,2$, then $X$ is $\alpha$-connected, so
it has 1 $\alpha$-end. If $\alpha\ge(q_1-1)/q_1$, but
$\alpha<(q_2-1)/q_2$ (or vice versa), then we can view $X$ as the
disjoint union of spaces isometric to $X_{q_1}$ (or $X_{q_2}$) with no
connection between them. This means that $X$ has infinitely
(countably) many $\alpha$-ends. Finally, if $\alpha<(q_i-1)/q_i$ for
$i=1,2$, then $X$ has no $\alpha$-ends.
\ee

\begin{example}\label{exfan}
As in Example~\ref{ex4}, we take points $z_m=(\cos(1/m),\sin(1/m))$ on
the unit circle in $\R^2$, where $m=1,2,3,\dots$, but now also take
$z_\infty=(1,0)$. Now we define $X=\{nz_m | m \le \infty,
n=0,1,2,3,\dots\}$. Then $X$ is coarsely connected (in fact
(0,1)-connected), and its space of 0-ends is its space of (0,1)-ends,
which is homeomorphic to the subspace $Z=\{z_m | m \le \infty\}$ of
the circle.
\ee

\begin{example}\label{exsparsefan}
As in Example~\ref{exfan}, we take $Z=\{z_m | m \le \infty\}$ where
$z_m= \break (\cos(1/m),\sin(1/m))$ on the unit circle in $\R^2$ for
$m=1,2,3,\dots$ and $z_\infty=(1,0)$. Let $X=\{nmz_m | m =1,2,3,
\dots, n=0,1,2,\dots\} \cup \{nz_\infty | n=0,1,2,\dots\}$. Then $X$
is coarsely connected (though not (0,K)-connected for any K), and, in
contrast to the previous example, its space of 0-ends is a countable
discrete space, which can be thought of as $Z$ but with the discrete
topology.
\ee

\end{document}